\documentclass{amsart}
\usepackage{amsmath}
\newtheorem{theorem}{Theorem}[section]
\newtheorem{lemma}[theorem]{Lemma}
\newtheorem{cor}[theorem]{Corollary}

\theoremstyle{definition}
 
\newtheorem{remark}[theorem]{Remark}

\newcommand{\N}{\mathbb{N}}
\newcommand{\p}{{\mbox{$[p]$}}}
\newcommand{\pd}{{\mbox{$[p]'$}}}
\newcommand{\scp}{{\mbox{$\scriptstyle [p]$}}}
\newcommand{\scpd}{{\mbox{$\scriptstyle [p]'$}}}
\DeclareMathOperator{\ad}{ad}
\DeclareMathOperator{\Hom}{Hom}
\DeclareMathOperator{\asoc}{ASoc}

\title{Faithful completely reducible representations of modular Lie algebras}
\author{Donald W. Barnes}
\address{1 Little Wonga Rd.\\Cremorne NSW 2090\\Australia\\}
\email{D.Barnes@maths.usyd.edu.au}


\begin{document}

\begin{abstract} Let $L$ be a Lie algebra of dimension $n$ over a field $F$ of characteristic $p > 0$.  I prove  the existence of a faithful completely reducible $L$-module of dimension less than or equal to $p^{n^2-1}$.
\end{abstract}

\maketitle

\section{Introduction}

Let $L$ be a Lie algebra of dimension $n$ over the field $F$.  The Ado-Iwasawa Theorem asserts that there exists a faithful finite-dimensional $L$-module $V$.  There are several extensions of this result which assert the existence of such a module $V$ with various additional properties. See, for example,  Hochschild \cite{Hoch}, Barnes \cite{extras}. Of importance for this paper is Jacobson's Theorem, \cite[Theorem 5.5.2]{SF} that every finite-dimensional Lie algebra $L$ over a field $F$ of characteristic $p>0$ has a finite-dimensional faithful completely reducible module $V$.  None of these results sets a bound to the dimension of $V$, unlike the Leibniz algebra analogue \cite{faithful} which asserts for a Leibniz algebra of dimension $n$, the existence of a faithful Leibniz module of dimension less than or equal to $n+1$.  This raises the question ``Is there an analogous strengthening of the Ado-Iwasawa Theorem?" that is, ``For a field $F$, does there exist a function $f:\N \to \N$ such that every Lie algebra of dimension $n$ over $F$ has a faithful module of dimension less than or equal to $f(n)$?"  The main purpose of this paper is to prove the following strengthening of Jacobson's Theorem, thereby  answering this question in the affirmative for fields $F$ of non-zero characteristic.

\begin{theorem}\label{main} Let $F$ be a field of characteristic $p>0$ and let $L$ be a Lie algebra of dimension $n$ over $F$.  Then $L$ has a faithful completely reducible module $V$ with $\dim(V) \le p^{n^2-1}$.
\end{theorem}

In all that follows, $F$ is a field of characteristic $p>0$ and $L$ is a Lie algebra of dimension $n$ over $F$.

\section{Restricted Lie algebras.}

A restricted Lie algebra (see \cite[Chapter 2]{SF}) is a Lie algebra $L$ together with a $p$-operation, that is, a map $\p: L \to L$ such that for $a,b \in L$ and $\lambda \in F$, we have $\ad(a^\scp) = \ad(a)^p$, $(\lambda a)^\scp = \lambda^p a^\scp$  and
$$(a+b)^\scp = a^\scp + b^\scp + \sum_{i=1}^{p-1}s_i(a,b),$$
where the $s_i(a,b)$ are defined by 
$$\bigl(\ad(a \otimes X+b \otimes 1)\bigr)^{p-1}(a \otimes 1) = \sum_{i=1}^{p-1}i s_i(a,b) \otimes X^{i-1}$$
in $L \otimes_F F[X]$.  

For convenience of reference, we list here some properties of $p$-operations.

\begin{lemma} \label{siab} Let $(L,\p)$ be restricted Lie algebra.  Then
\begin{enumerate}
\item If $[a,b]=0$, then $[a^{\scp^r},b^{\scp^s}]=0$.  In particular, $[a^{\scp^r}, a^{\scp^s}]=0$.
\item  If $[a,b]=0$, then $(a+b)^\scp = a^\scp + b^\scp$.
\item For all $a,b \in L$, we have $s_i(a,b) \in L'$.
\end{enumerate}
\end{lemma}

\begin{proof} 
(1) Since $\ad(a)b=0$,  $\ad(a)^{p^r}b=0$, that is, $[a^{\scp^r}, b]=0$.  But $[b,a^{\scp^r}] = 0$ implies that $[b^{\scp^s}, a^{\scp^r} ]=0$.

(2) Since $\ad(a \otimes X+b \otimes 1)(a \otimes 1) = 0$, we have $s_i(a,b)=0$ for all $i$.

(3) Follows immediately from the definition.  (In fact, by \cite[Lemma 2.1.2]{SF}, $s_i(a,b) \in L^p$.) 
\end{proof}

\begin{lemma} \label{abid} Let $(L,\p)$ be a restricted Lie algebra and let $A$ be an abelian ideal of $L$.  Then there exists a $p$-operation $\pd$ on $L$ such that $a^\scpd = 0$ for all $a \in A$.
\end{lemma}

\begin{proof} Take a basis $\{a_1, \dots, a_r\}$ of $A$ and extend this to a basis $\{a_1, \dots, a_n\}$ of $L$.  Put $b_i = a_i^\scp$.  For $i=1, \dots, r$, we have $\ad(a_i)^2=0$.  We replace these $b_i$ with $0$.  By Jacobson's Theorem \cite[Theorem 2.2.3]{SF}, there exists a $p$-operation $\pd$ on $L$ such that $a_i^\scpd = 0$ for $i = 1, \dots, r$ (and $a_j^\scpd = b_j$ for $j > r$). From Lemma \ref{siab}(2), it then follows that $a^\scpd = 0$ for all $a \in A$.  
\end{proof}

\begin{cor} \label{pid} Let $(L,\p)$ be a restricted Lie algebra.  Then there exists a $p$-operation $\pd$ on $L$ such that every abelian minimal ideal of $L$ is a \pd-ideal.
\end{cor}

\begin{proof}  The abelian socle $\asoc(L)$ is the sum of all the abelian minimal ideals of $L$.  It is an abelian ideal.  By Lemma \ref{abid}, there exists a $p$-operation $\pd$ which is zero on $\asoc(L)$, and so, on every abelian minimal ideal.  Thus every abelian minimal ideal is a \pd-ideal.
\end{proof}

\begin{theorem} \label{res} Let $(L, \p)$ be a restricted Lie algebra of dimension $n$ over the field $F$ of characteristic $p$.  Then $L$ has a faithful completely reducible module of dimension less than or equal to $p^{n-1}$.
\end{theorem}

\begin{proof} By Corollary \ref{pid}, we may suppose that every abelian minimal ideal is a \p-ideal.  The result holds for $n=1$.  We use induction over $n$. 

 Suppose that $A_1, A_2$ are distinct minimal \p-ideals of $L$.  Then $L/A_i$ is a restricted Lie algebra.  Since $\dim(L/A_i) \le n-1$, $L/A_i$ has a faithful completely reducible module $V_i$ with $\dim(V_i) \le p^{n-2}$.  But $V_1 \oplus V_2$ is a faithful completely reducible $L$-module and $\dim(V_1 \oplus V_2) \le 2p^{n-2} \le p^{n-1}$.

Suppose that $A$ is the only minimal \p-ideal of $L$.  Let $B \subseteq A$ be a minimal ideal of $L$.  The representation of $L$ on $B$ is a \p-representation and its kernel $K$ is a \p-ideal. Either $K=0$ or $K \supseteq A$.    If $K=0$, then $B$ is a faithful completely reducible $L$-module and the result holds.  So we may suppose that $K \supseteq A$.  But this implies that $B$ is abelian.  By our choice of \p, this implies that $B$ is a \p-ideal and so, that $B=A$.  Hence we may assume that our only minimal \p-ideal $A$ is also a minimal ideal and is abelian.  

We can take a linear map $c : L \to F$ such that $c(A) \ne 0$. Let $W = \langle w \rangle$ be the $1$-dimensional $A$-module with the action $aw = c(a)w$ for all $a \in A$.  Then $W$ has character $c|A$.   We form the $c$-induced module $V = u(L,c) \otimes_{u(A,c|A)} W$. See \cite[Chapter 5]{SF}.  By \cite[Proposition 5.6.2]{SF}, $\dim(V) = p^{\dim(L/A)} \le p^{n-1}$.  Let $V_0$ be the direct sum of the composition factors of $V$.  Then $\dim(V_0) \le p^{n-1}$.  Note that $A$ acts non-trivially on $V_0$ since it acts non-trivially on the irreducible $A$-submodule $1 \otimes W$ of $V$.  If $V_0$ is faithful, the result holds.

Let $\{e_1, \dots, e_k\}$ be a co-basis of $A$ in $L$.  Then by \cite[Proposition 5.6.2]{SF}, the $e_1^{r_1}e_2^{r_2} \dots e_k^{r_k} \otimes w$ with the $r_i <p$ form a basis of $V$.  For $x = \sum \lambda_i e_i +a$ with $a \in A$, $x(1 \otimes w) = \sum \lambda_i e_i \otimes w + 1 \otimes aw$.  If $x(1 \otimes w)=0$ then we must have $\lambda_i = 0$ for all $i$, that is, $x \in A$.  Thus the representation of $L$ on $V$ has kernel $\ker(V) \subseteq A$.  As $A$ is a minimal ideal and acts non-trivially, we have $\ker(V) = 0$.  Thus $V$ is faithful. 

So suppose that $V_0$ is not faithful.  Then there exists a minimal ideal $B$ whose action on every composition factor of $V$ is trivial.  Then $B$ is represented on $V$ by nilpotent linear transformations.  But $V$ is faithful, so by Engel's Theorem for algebras of linear transformations, $B$ is nilpotent.  But $B'$ is an ideal of $L$, so we must have $B'=0$.  By our choice of \p, $B$ is a \p-ideal of $L$, contrary to $A$ being the only minimal \p-ideal of $L$.  Therefore $V$ is faithful.
\end{proof}

\section{Minimal $p$-envelopes.}
Let $(L^e,\p)$ be a minimal $p$-envelope of $L$.  We investigate $\dim(L^e)$.  Note that, by \cite[Theorem 2.5.8(1)]{SF}, $\dim(L^e)$ is independent of the choice of minimal $p$-envelope.  Let $Z$ be the centre of $L^e$.  By \cite[Theorem 2.5.8(3)]{SF}, $Z \subseteq L$.  By \cite[Proposition 2.1.3(2)]{SF},   $(L^e)' \subseteq L$.

\begin{lemma} \label{idL} Let $A$ be an ideal of $L$.  Then $A$ is an ideal of $L^e$.
\end{lemma}

\begin{proof} The set $\{x\in L^e \mid \ad(x)A \subseteq A\}$ is a \p-subalgebra of $L^e$ and contains $L$.
\end{proof}

\begin{lemma} \label{sums} Let $a_1, \dots, a_r \in L^e$ and $\lambda_1, \dots , \lambda_r \in F$.  Then
$$ (\sum_{i=1}^r \lambda_i a_i)^\scp = \sum_{i=1}^r \lambda_i^p a_i^\scp + k$$
for some $k \in L$. 
\end{lemma}

\begin{proof} From the definition of a $p$-operation, we have $(\lambda_i a_i)^\scp = \lambda_i^p a_i^\scp$.  The result holds for $r=2$ by Lemma \ref{siab}(3) since $(L^e)' \subseteq L$.  So $(\lambda_1 a_1 + \dots + \lambda_r a_r)^\scp = (\lambda_1 a_1 + \dots \lambda_{r-1} a_{r-1})^\scp + \lambda_r^p a_r^\scp +k_1$ for some $k_1 \in L$.  But by induction, $(\lambda_1 a_1 + \dots + \lambda_{r-1} a_{r-1})^\scp = \lambda_1^p a_1^\scp+ \dots + \lambda_{r-1}^p a_{r-1}^\scp +k_2$ for some $k_2 \in L$.  The result follows.
\end{proof}

\begin{lemma} \label{powers} Let $x \in L$ and let $V= \langle x^{\scp^i} \mid i=1,2, \dots \rangle$ be the space spanned by the $x^{\scp^i}$.  Then $\dim((V+L)/L) \le n$.
\end{lemma}

\begin{proof}

We have $\ad(x)L^e \subseteq L$.  The maps $\ad(x^{\scp^i})| L \to L$ are powers of $\ad(x)|L$.  So they span a subspace of $\Hom(L,L)$ of dimension at most $n$.  For some $r \le n-1$, the maps $\ad(x)|L, \ad(x^\scp)|L, \dots, \ad(x^{\scp^r})|L$ are linearly independent with 
$$\ad(x^{\scp^{r+1}})|L = \sum_{i=0}^r \lambda_i \ad(x^{\scp^i})|L$$
for some $\lambda_i \in F$.  Put $y = x^{\scp^{r+1}} - \sum_{i=0}^r \lambda_i x^{\scp^i}$. Then $\ad(y)L^e \subseteq L$ and $\ad(y)L = 0$.  Thus $\ad(y)^pL^e = 0$ and it follows that $y^\scp \in Z \subseteq L$.

By Lemma \ref{siab}(1) and (2),  $y^\scp = x^{\scp^{r+2}} - \sum_{i=0}^r \lambda_i^p x^{\scp^{i+1}}$.  Thus $x^{\scp^{r+2}} \in \langle x^\scp, \dots, x^{\scp^{r+1}} \rangle + Z$. Suppose that  $x^{\scp^{r+s}} \in \langle x^\scp, \dots, x^{\scp^{r+1}} \rangle + Z$.  Then $x^{\scp^{r+s}} = \mu_1x^\scp + \dots + \mu_{r+1}x^{\scp^{r+1}} +z$ for some $\mu_i \in F$ and $z \in Z$.  By Lemma \ref{siab}(1) and (2), $x^{\scp^{r+s+1}} =  \mu_1^p x^{\scp^2} + \dots + \mu_{r+1}^px^{\scp^{r+2}} +z^\scp$.  Since $x^{\scp^{r+2}} \in \langle x^\scp, \dots, x^{\scp^{r+1}} \rangle + Z$ and $z^\scp \in Z$, we have $x^{\scp^{r+s+1}} \in \langle x^\scp, \dots, x^{\scp^{r+1}} \rangle + Z$.  It follows by induction over $s$, that $\langle x^\scp, \dots, x^{\scp^{r+1}} \rangle + Z = V+Z$ and so, that $\dim((V+L)/L) \le r+1 \le n$.
\end{proof}

\begin{theorem} \label{env} Let $L$ be a Lie algebra of dimension $n$ over the field $F$ of characteristic $p>0$ and let $A$ be an abelian ideal of $L$ with $\dim(A)=d$.  Let $(L^e, \p)$ be a minimal $p$-envelope of $L$. Then $\dim(L^e) \le n(n-d+1)$.
\end{theorem}

\begin{proof} We choose a basis $\{e_1, \dots, e_n\}$ of $L$ with $e_{n-d+1}, \dots, e_n \in A$.  By Lemma \ref{abid}, we may suppose that $a^\scp = 0$ for all $a \in A$.  Then the $e_i^{\scp^j} = 0$ for $i > n-d$ and $j > 0$.  For each $i$, let $V_i$ be the subspace of $L^e$ spanned by  the $e_i^{\scp^j}$ (including $j=0$) and let $V = \sum_iV_i$.   Then $V \supseteq L$ and $V/L = \sum_{i=1}^{n-d} (V_i+L)/L$.  By Lemma \ref{powers}, $\dim(V_i+L/L) \le n$, so  $\dim(V/L) \le n(n-d)$ giving $\dim(V) \le n(n-d+1)$.   But $(L^e)' \subseteq L$, so $V$ is a subalgebra of $L^e$.  By Lemma \ref{sums}, $v^\scp \in V$ for all $v \in V$.  Thus $V$ is a $p$-envelope of $L$, so $V = L^e$. \end{proof}

\section{The main result.}

\begin{proof}[Proof of Theorem \ref{main}]  We use induction over $n$.  Suppose that $A_1,A_2$ are distinct minimal ideals of $L$.  Then $L/A_i$ has a faithful completely reducible module $V_i$ with $\dim(V_i) \le p^{(n-1)^2 -1}$ and $V_1 \oplus V_2$ is a module satisfying all the requirements.  So suppose that $A$ is the only minimal ideal of $L$.  If $A$ is non-abelian, then $A$ is an $L$-module satisfying the requirements, so suppose that $A$ is abelian.

  We take a minimal $p$-envelope $(L^e,\p)$ of $L$.  As $\dim(A) \ge 1$, by Theorem \ref{env}, we have $\dim(L^e) \le n^2$.  By Theorem \ref{res}, $L^e$ has a faithful completely reducible module $V$ with $\dim(V) \le p^{n^2-1}$.  There is some irreducible summand $V_0$ of $V$ on which $A$ acts non-trivially.  By Lemma \ref{idL}, $A$ is an ideal of $L^e$ and it follows that  $V_0^A := \{v \in V_0 \mid Av=0\}$ is an $L^e$-submodule of $V_0$.  Therefore $V_0^A =0$.  Let $V_1$ be an irreducible $L$-submodule of $V_0$.  Since $V_1^A \subseteq V_0^A$, we have $V_1^A = 0$. But $A$ is the only minimal ideal of $L$. As it is not in the kernel of the representation of $L$ on $V_1$,  $V_1$ is a faithful $L$-module.  
\end{proof}

\begin{remark} We have a function $f: \N \to \N$, namely $f(n)=p^{n^2-1}$, such that every Lie algebra of dimension $n$ over a field of characteristic $p$ has a faithful completely reducible module of dimension less than or equal to $f(n)$.  We cannot replace this with a function independent of $p$, for suppose that $f:\N \to \N$ were claimed to be such a function.  The smallest faithful completely reducible module for the non-abelian algebra of dimension $2$ has dimension $p$, so this algebra over a field of characteristic $p > f(2)$ is a counterexample.  This does not rule out the possibility, if we drop the requirement of complete reducibility, of there being a function $f$ independent of $p$ such that every Lie algebra of dimension $n$ over a field of non-zero characteristic has a faithful module of dimension less than or equal to $f(n)$.
\end{remark}

It is not claimed that any of the bounds given in this paper are best possible.

\bibliographystyle{amsplain}

\end{document}